\documentclass[12pt]{amsart}

\usepackage{cite}

\usepackage[round]{natbib}

\usepackage{kotex}
\usepackage{lineno,hyperref}
\usepackage{epstopdf}
\usepackage{pifont}
\usepackage{dsfont}
\usepackage{tikz}
\usepackage{latexsym}
\usepackage{graphics}
\usepackage{graphicx}
\usepackage{float}
\usepackage[dvips]{xy}
\xyoption{all}
\usepackage{ifpdf}
\usepackage{multirow}
\usepackage{amssymb, amsthm, amsmath}
\usepackage{hhline}
\usepackage{centernot}
\usepackage{verbatim, comment, color}
\usepackage{enumerate}
\usepackage{setspace}
\usepackage{array}
\usepackage{tabularx}

\newcolumntype{b}{X}
\newcolumntype{s}{>{\hsize=.5\hsize}X}

\usepackage{lipsum}
\usepackage{subcaption}
\usepackage{mathtools}

\usepackage{enumitem}

\usepackage{kbordermatrix}

\newtheorem{theorem}{Theorem} 

\theoremstyle{definition}
\newtheorem{example}{Example}[section]

\theoremstyle{remark}

\newcommand{\ba}{\begin{align}}
\newcommand{\ea}{\end{align}}
\newcommand{\be}{\begin{equation}}
\newcommand{\ee}{\end{equation}}

\newcommand{\bea}{\begin{eqnarray}}
\newcommand{\eea}{\end{eqnarray}}
\newcommand{\barr}{\begin{array}}
\newcommand{\earr}{\end{array}}
\newcommand{\bn}{\begin{enumerate}}
\newcommand{\en}{\end{enumerate}}
\newcommand{\bi}{\begin{itemize}}
\newcommand{\ei}{\end{itemize}}
\newcommand{\bbbm}{\begin{pmatrix}}
\newcommand{\eeem}{\end{pmatrix}}

\newcommand{\cE}{{\cal E}}
\newcommand{\cF}{{\cal F}}
\newcommand{\cG}{{\cal G}}

\newcommand{\cI}{{\cal I}}
\newcommand{\cN}{{\cal N}}
\newcommand{\cP}{{\cal P}}

\newcommand{\cV}{{\cal V}}

\newcommand{\E}{{\mathbb E}}

\newcommand{\N}{{\mathbb N}}

\newcommand{\R}{{\mathbb R}}

\newcommand{\al}{\alpha}

\newcommand{\tta}{\theta}
\newcommand{\ignore}[1]{}{}
\newcommand{\noin}{\noindent}

\newcommand{\nn}{\nonumber}

\newcommand{\p}{{\partial}}

\newcommand{\q}{\quad}

\newcommand{{\QED}}{{\hfill QED} \bigskip}

\renewcommand{\subset}{\subseteq}

\newcommand{\cal}{\mathcal}

\newcommand{\es}{\varnothing}

\definecolor{darkspringgreen}{rgb}{0.09, 0.45, 0.27} 
\definecolor{darkgray}{rgb}{0.66, 0.66, 0.66}

\numberwithin{equation}{section}
\numberwithin{theorem}{section}
\tikzset{
    partial ellipse/.style args={#1:#2:#3}{
        insert path={+ (#1:#3) arc (#1:#2:#3)}
    }
}

\usepackage[bottom=3.2cm, right=2.7cm, left=2.7cm, top=3.2cm]{geometry}

\begin{document}

\title[Foundations for the Hodge-Theoretic Shapley Value]
{Axiomatic and Probabilistic Foundations for the Hodge-Theoretic Shapley Value
}
\thanks{
}

\date{\today}

\author{Tongseok Lim}
\address{Tongseok Lim: Mitchell E. Daniels, Jr. School of Business
\newline  Purdue University, West Lafayette, Indiana 47907, USA}
\email{lim336@purdue.edu}

\begin{abstract} 

This paper establishes a complete theoretical foundation for the Hodge-theoretic extension of the Shapley value introduced by Stern and Tettenhorst (2019). We show that a set of five axioms---efficiency, linearity, symmetry, a modified null-player condition, and an independency principle---uniquely characterize this value across all coalitions, not just the grand coalition. In parallel, we derive a probabilistic representation interpreting each player's value as the expected cumulative marginal contribution along a random walk on the coalition graph. These dual axiomatic and probabilistic results unify fairness and stochastic interpretation, positioning the Hodge-theoretic value as a canonical generalization of Shapley's framework.

\end{abstract}

\maketitle

\noindent\emph{Keywords: Cooperative game, Shapley axiom, Shapley value, Shapley formula, Hodge theory, Poisson's equation, path integral, random coalition formation process
}

\noindent\emph{JEL classification:
C71. MSC2020 Classification: 91A12, 05C57, 60J20, 68R01}


\section{Introduction}

Lloyd Shapley's value allocation theory for cooperative games has been one of the most central concepts in game theory. The Shapley value is widely used in many fields, including economics, finance, and machine learning, to allocate resources, assess individual agent contributions, and determine the fairness of payouts. Its applications are vast and continue to expand; modern treatises discuss its use in genetics, social networks, finance, politics, telecommunication, and operations-research problems like queueing and aircraft landing fees \citep{algaba2019handbook}. Recently, researchers have started to utilize the Shapley value in diverse fields such as machine learning for data valuation and feature attribution \citep{ghorbani2019data, mitchell2022sampling, schoch2022cs, rozemberczki2022shapley}, medicine for interpreting model predictions \citep{rodriguez2019interpretation, smith2021identifying}, and sustainable energy for cost allocation \citep{pang2021correlation}. This shows that Shapley's cooperative value allocation theory remains a vibrant area of research, applied across various contexts, and continues to inspire researchers.

The enduring appeal of the Shapley value stems from two notable facets: its four defining axioms and its elegant value allocation formula. The axioms—efficiency, symmetry, null-player, and linearity—establish a framework of fairness criteria for evaluating the individual contributions of players to a cooperative game. Since the Shapley value is the unique outcome satisfying these axioms, it emerges as a compelling solution for fair allocation. Furthermore, the Shapley formula provides a concrete mathematical method for computing each player's value. This formula calculates a player's expected marginal contribution by averaging over all possible orders in which players can join to form the grand coalition. Endowed with desirable properties inherited from the axioms, the Shapley formula is widely embraced as the standard approach for distributing the value of a cooperative game. 

A key assumption in the classic framework is that all players will eventually form the grand coalition, and the axioms are used to determine a fair allocation of the total value $v(N)$. Consequently, the theory does not directly address how to assess player contributions when the game concludes in a partial coalition state, i.e., a coalition $S \subsetneq N$. While one can apply the Shapley formula to each subgame $v$ restricted to a coalition $S$, this approach implicitly assumes the coalition grows only towards the target $S$, thereby failing to capture the full structure of the larger game involving all players in $N$. 

Addressing this gap, \citet{stern2019hodge} proposed a novel value concept based on the graph Poisson's equation, rooted in combinatorial Hodge theory \citep{lim2020hodge, candogan2011flows, jiang2011statistical}. Their framework defines a value for each player at every possible coalition state and shows that their value satisfies a natural extension of the Shapley axioms and, importantly, recovers the classic Shapley value for the grand coalition. 

However, this innovative work leaves two critical questions unanswered. First, while the classic Shapley axioms fully characterize the Shapley value, the properties presented by Stern and Tettenhorst do not uniquely characterize their value for all partial coalitions. This lack of a complete axiomatic foundation makes it difficult to argue for its canonicity. Second, it is unclear if their value admits a probabilistic interpretation analogous to the Shapley formula, which represents a probabilistic average of a player's marginal contribution. 

The primary objective of this study is to address these fundamental questions by establishing a complete theoretical foundation for this Hodge-theoretic value concept. We aim to (1) formulate a set of axioms that uniquely determines the value allocation for every possible coalition, not just the grand coalition, and (2) derive a probabilistic formula that, analogous to the classic Shapley formula, interprets the value in terms of a player's expected contribution during a dynamic coalition formation process. In the tradition of seeking alternative and deeper foundations for value allocation concepts, our work provides both axiomatic and probabilistic underpinnings that solidify this extension as a natural and robust generalization of Shapley's original theory.

This paper is structured as follows. Section \ref{Shapleyreview} provides a brief review of the classic Shapley value, its axioms, and its formula. Section \ref{STsection} details the Hodge-theoretic extension proposed by Stern and Tettenhorst. Section \ref{newaxiom} presents our first main result: a new set of five axioms that uniquely characterizes their value. Section \ref{newformula} introduces our second main result: a probabilistic representation of the value as a path integral over a random walk on the coalition graph. Section \ref{conclusion} concludes the paper with a summary and implication of our findings. Proofs of the results are presented in Sections \ref{proofs}.

\section{Review of Shapley axioms and the Shapley formula}\label{Shapleyreview}
We commence by revisiting the renowned Shapley value allocation theory (\citet{shapley1953value}), which remains a source of inspiration for researchers across diverse fields. To begin, let $N = \{1,2,...,n\}$ represent the set of {\em players} of the {\em coalition games}
\be
\cG_N = \{ v : 2^{N} \to \R \mid v(\varnothing)=0 \}. \nn 
\ee
A coalition game $v$ is a function on the subsets of $N$, where each  $S \subset N$ represents a coalition of players in $S$, and $v(S)$ represents the value assigned to the coalition $S$, with the null coalition $\varnothing$ receiving zero value. Given $v \in \cG_N$, Shapley considered the question of how to split the grand coalition value $v(N)$. The resulting allocation, known as the Shapley value, is determined uniquely by the following result. 
\begin{theorem}[\citet{shapley1953value}]
  \label{thm:shapley}
  There exists a unique allocation
  $ v \in \cG_{N} \mapsto \bigl( \phi _i (v) \bigr) _{ i \in N } $ satisfying the following conditions:
\vspace{1mm}
  
\noin{\rm {\boldmath$\cdot$} efficiency:} $ \sum _{ i \in N } \phi _i (v) = v (N) $.  \vspace{1mm}

\noin{\rm {\boldmath$\cdot$} symmetry:} 
    $ v \bigl( S \cup \{ i \} \bigr) = v \bigl( S \cup \{ j \} \bigr)
    $ for all $ S \subset N \setminus \{ i, j \} $ yields
 $ \phi _i (v) = \phi _j (v) $. \vspace{1mm}

\noin{\rm {\boldmath$\cdot$} null-player:} 
    $ v \bigl( S \cup \{ i \} \bigr) - v (S) = 0 $ for all
 $ S \subset N \setminus \{ i \} $ yields $ \phi _i (v) = 0 $. \vspace{1mm}

\noin{\rm {\boldmath$\cdot$} linearity:} 
    $ \phi _i ( \alpha v + \alpha ^\prime v ^\prime ) = \alpha \phi _i
    (v) + \alpha ^\prime \phi _i ( v ^\prime ) $ for all
 $ \alpha, \alpha ^\prime \in \mathbb{R} $ and $v,v' \in \cG_{N}$. \vspace{1mm}

  Moreover, this allocation is given by the following explicit formula:
  \begin{equation}
    \label{eqn:shapley}
    \phi _i (v) = \sum _{ S \subset N \setminus \{ i \} } \frac{ \lvert S \rvert ! \bigl( |N| - \lvert S \rvert -1 \bigr) ! }{ |N| ! } \Bigl( v \bigl(  S \cup \{ i \} \bigr) - v (S) \Bigr).   \end{equation}
\end{theorem}
The four conditions represent distinct fairness criteria. Specifically, [efficiency] indicates that the value obtained by the grand coalition is fully distributed among the players; [symmetry] indicates that equivalent players receive equal amounts; [null-player] indicates that a player who contributes no marginal value to any coalition receives nothing; and [linearity] indicates that the allocation is linear in terms of game values. These four conditions are known as the \emph{Shapley axioms}, the vector $\bigl( \phi _i (v) \bigr) _{ i \in N } $ is referred to as the {\em Shapley value}, and \eqref{eqn:shapley} is denoted as the {\em Shapley formula}. 

The Shapley formula \eqref{eqn:shapley} has a powerful probabilistic interpretation. Assume the players form the grand coalition one at a time, following an order determined by a random permutation $\sigma$ of $N$, such that at each stage every remaining player has an equal chance of being the next to join. In any such sequence, player $i$ joins the coalition $ S^{\sigma}_i= \bigl\{ j \in N : \sigma (j) < \sigma (i) \bigr\}$ that has already formed, contributing the marginal value $ v \bigl( S^{\sigma}_i \cup \{ i \} \bigr) - v (S^{\sigma}_i) $.  Then $ \phi _i (v) $ is precisely the average marginal value contributed by player $i$ over all 
possible permutations: 
\begin{align}
  \label{eqn:shapleyPermutation}
  \phi _i (v) = \frac{ 1 }{ |N| ! } \sum _\sigma \Bigl( v \bigl( S^{\sigma}_i \cup \{ i \} \bigr) - v (S^{\sigma}_i ) \Bigr). 
\end{align}
The well-known glove game below illustrates this formula in a simple context.
\begin{example}[Glove game]\label{ex:introGlove}
Let $ |N| = 3 $. Suppose player $1$ has a left-hand glove, while players $2$ and $3$ each have a right-hand glove. A pair of gloves has value $1$, while unpaired gloves have no value. This defines the game where $ v (S) = 1 $ if $S$ contains player $1$ and at least one of players $2$ or $3$, and $ v (S) = 0 $ otherwise. The Shapley values are:
  \begin{equation*}
\phi _1 (v) = \tfrac{ 2 }{ 3 } , \qquad \phi _2 (v) = \phi _3 (v) = \tfrac{ 1 }{ 6 }.
  \end{equation*}
This is easily seen from \eqref{eqn:shapleyPermutation}: player $1$ contributes marginal value $0$ only when joining first (which happens in 2 of the 6 permutations) and marginal value $1$ otherwise (in the other 4 permutations), so $ \phi _1 (v) = \frac{4}{6} = \tfrac{ 2 }{ 3 } $. Efficiency and symmetry then yield $\phi _2 (v) = \phi _3 (v) = \tfrac{ 1 }{ 6 }$. 
\end{example}
We note that the Shapley value can be readily applied to each coalition $T \subset N$ by employing the Shapley formula to the subgame $v \big|_T$\footnote{$v \big|_T : 2^T \to \R$ denotes the restriction of $v$ to the subsets of $T$, i.e., $v \big|_T(S) = v(S)$ for all $S \subset T$.}. This allocation scheme is termed the {\em extended Shapley value}. In the subsequent sections, we explore an alternative cooperative value allocation scheme that seamlessly extends to all coalitions $2^N$ and differs from the Shapley value for partial coalitions $T \subsetneq N$. 

\section{Stern-Tettenhorst's extension of the Shapley value via the graph Poisson's equation}\label{STsection}

Consider the (undirected) hypercube graph, or coalition game graph $G=(\cV, \cE)$, where $\cV$ denotes the set of nodes and $\cE$ the set of edges. This graph is defined by
\begin{equation}\label{oldG}
  \cV := 2 ^{N} = \{ S \mid S \subset N\}, \ \  \cE := \bigl\{ \bigl(  S, S \cup \{ i \} \bigr) \in \cV \times \cV \ | \  S \subset N \setminus \{ i \} ,\ i \in N \bigr\}. 
\end{equation}
Notice that each coalition $S \subset N$ corresponds to a vertex of the unit hypercube in $\R^{N}$. We assume that each edge is oriented in the direction of set inclusion $ S \hookrightarrow S \cup \{ i \}$. We also define the set of reverse (negatively-oriented) edges (see Figure \ref{figure1}):
\be
\cE_- := \bigl\{ \bigl(S \cup \{ i \}, S \bigr) \in \cV \times \cV \ | \  S \subset N \setminus \{ i \} ,\ i \in N \bigr\}. 
\ee
The edges in $\cE$ are termed forward or positively-oriented edges. We set $\overline \cE = \cE \cup \cE_-$.
\begin{figure}[h]
\centering
\begin{subfigure}{.5\textwidth}
  \centering
  \includegraphics[width=.8\linewidth]{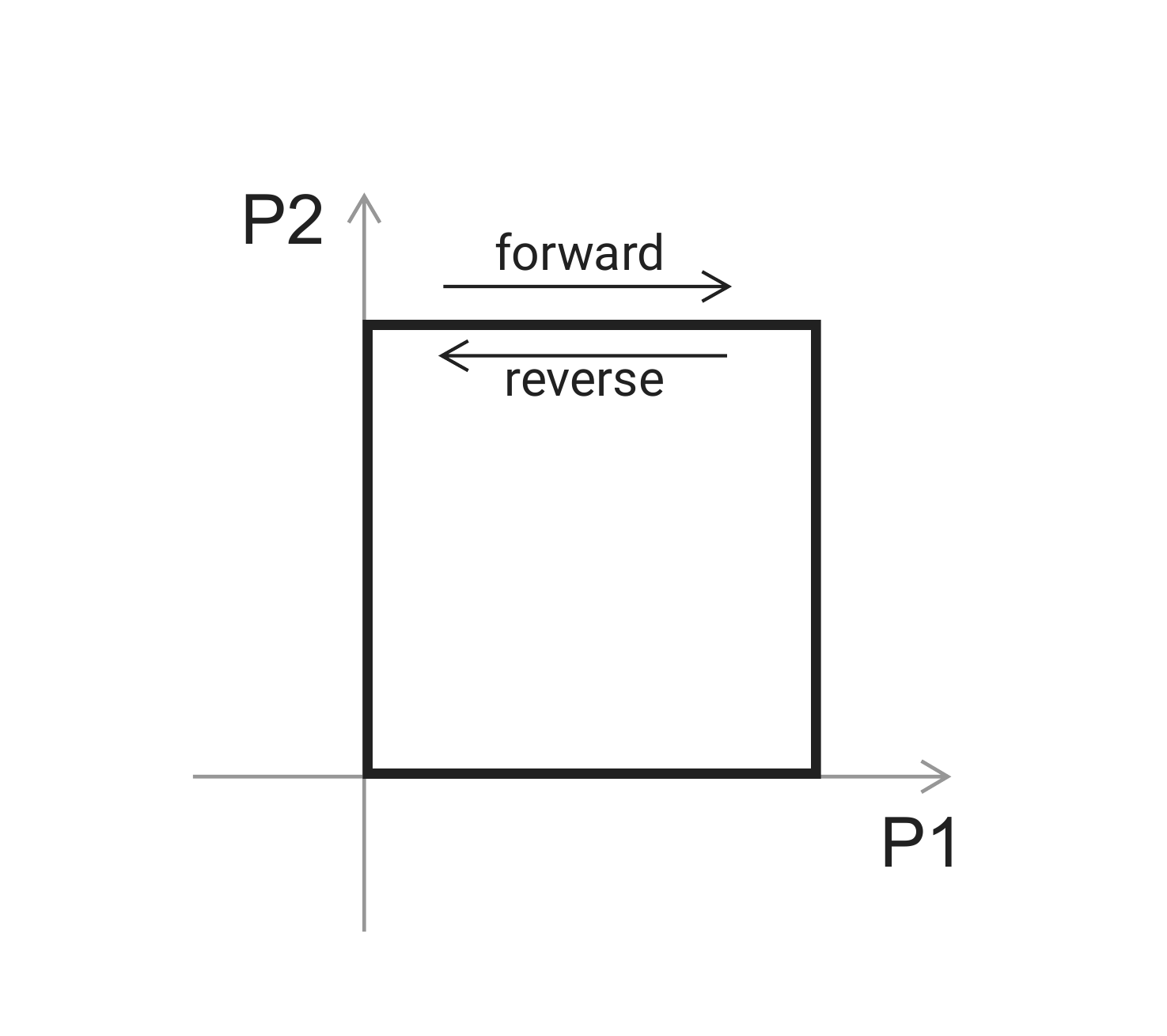}
\end{subfigure}%
\begin{subfigure}{.5\textwidth}
  \centering
  \includegraphics[width=.8\linewidth]{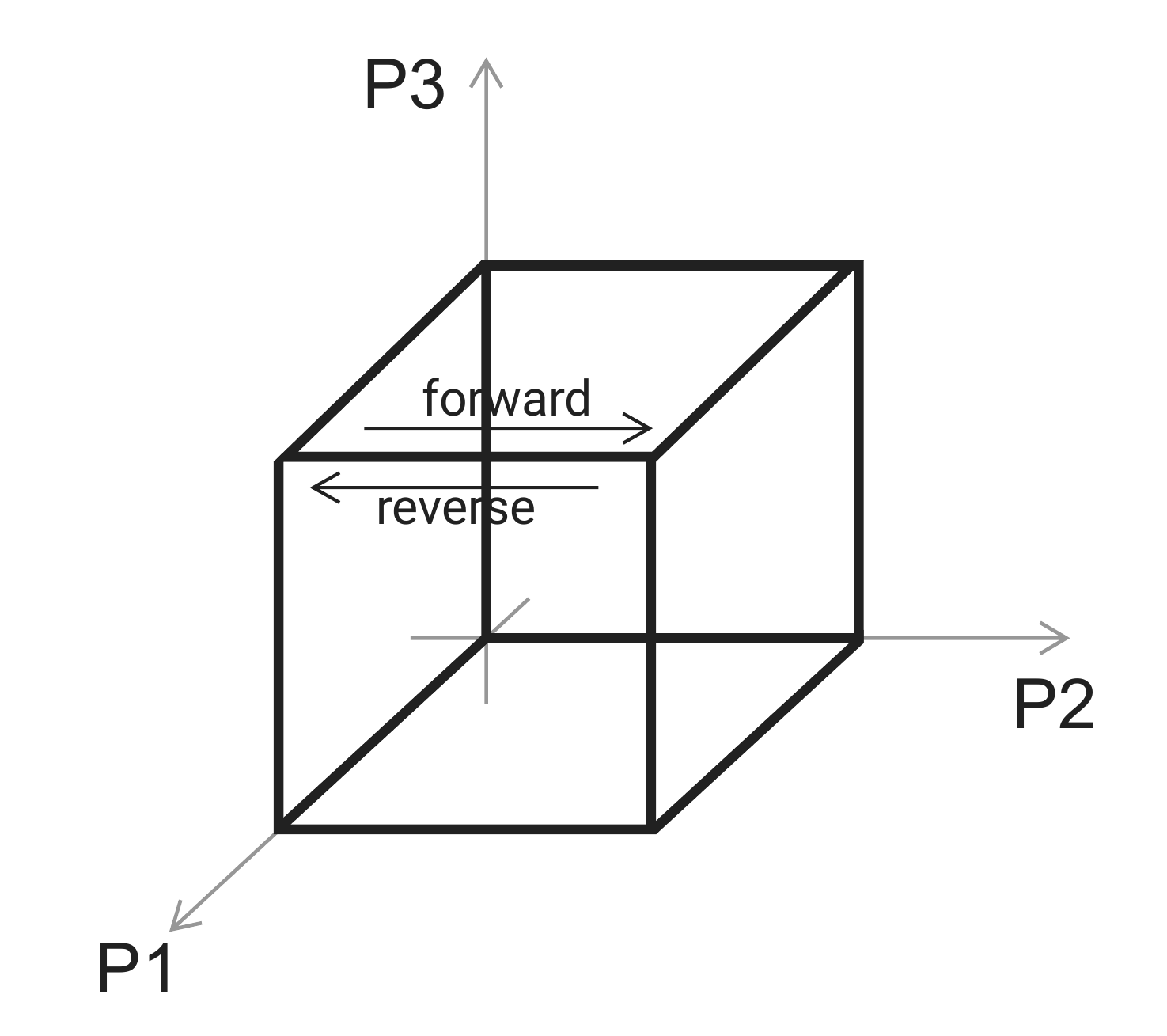}
\end{subfigure}
\caption{Coalition game graphs for $N=2$ and $N=3$. Each vertex of the cube corresponds to a coalition. The vertex $(1,0,1)$, for example, corresponds to the coalition $\{1,3\}$, and $(0,1,1)$ corresponds to $\{2,3\}$.} 
\label{figure1}
\end{figure}

For each game $v \in \cG_N$ and coalition $S \subset N$, let $\Phi_i(v,S)$ represent the value assigned to player $i$ when the players constitute coalition $S$. We assume $\Phi_i(v, \varnothing) = 0$ for all $i \in N$. Unlike the extended Shapley value, we do {\em not} require $\Phi_i(v,S) = 0$ for $i \notin S$. This allows players outside a coalition to be assigned a nonzero value, either positive or negative. For instance, if $N$ is the set of founding members of a company and a subset $S$ wishes to continue managing it, $S$ might need to provide compensation to the departing members $N \setminus S$, resulting in a nonzero value for them. This value could even be negative if a departing member caused damage to the company.

\citet{stern2019hodge} propose that for any $S \in 2^N$, the values $\Phi_i(v,S)$ must satisfy the following system of equations:
\begin{align}\label{STPoisson}
&\Phi_i(v,S) - \frac{1}{|N|} \sum_{T \sim S} \Phi_i(v,T) = \frac{1}{|N|} \sum_{T \sim S} \p_i v(T,S),
\end{align}
where $T\sim S$ indicates that $T$ is an adjacent coalition to $S$ in the graph $G$, and where for each $ i \in N $, the {\em partial differential} $\p_i v : \overline \cE \to \R$ of a game $v$ is defined as 
  \begin{equation}\label{partialdifferential}
    \p_i v \bigl( S , S \cup \{ j \} \bigr) :=
    \begin{cases}
v (S \cup \{ i \}) - v(S) & \text{if } \ j=i, \\
      0 & \text{if } \  j \neq i,
    \end{cases}
    \end{equation} 
and $\p_i v \bigl(S \cup \{ j \}, S \bigr) := - \p_i v \bigl( S , S \cup \{ j \} \bigr)$. Thus, $ \p_i v$ represents the marginal contribution of player $i$ to the game $v$, accounting for both joining and leaving a coalition. Equation \eqref{STPoisson} asserts that the deviation of player $i$'s payoff from the average of its neighbors must equal the average marginal contribution made by that player across the same neighborhood. Note that \eqref{STPoisson} is a form of Poisson's equation, since its left-hand side is proportional to $L \Phi_i (S)$, where $L$ is the graph Laplacian on $G$.  The main result of \citep{stern2019hodge} is the following theorem, where they denote $v_i(S) = \Phi_i(v,S)$ and call the function $v_i: 2^N \to \R$ a {\em component game}. 
\begin{theorem}[\cite{stern2019hodge}, Theorem 3.4]
  \label{thm:gameDecomp}
The component games $ (v_i)_{i \in N}$ solving \eqref{STPoisson} with $v_i(\varnothing)=0$ satisfy the following:
  \begin{enumerate}[label=(\alph*)]
\item $ \displaystyle\sum _{ i \in N } v _i = v $. 
\item If $ v \bigl( S \cup \{ i \} \bigr) - v (S) = 0 $ for all $ S \subset N \setminus \{ i \} $, then $ v _i = 0 $. 
\item If $ \sigma $ is a permutation of $N$ and $ \sigma ^\ast v $ is the game defined by $ ( \sigma ^\ast v ) (S) = v \bigl( \sigma (S) \bigr) $, then $ ( \sigma ^\ast v ) _i = \sigma ^\ast ( v _{ \sigma (i) } ) $.
In particular, if $ \sigma $ is a permutation swapping $i$ and $j$ and if $ \sigma ^\ast v = v $, then $ v _i = \sigma ^\ast (v _j) $.
\item For any two games $ v , v ^\prime $ and $ \alpha, \alpha ^\prime \in \mathbb{R} $, $ ( \alpha v + \alpha ^\prime v ^\prime ) _i = \alpha v _i + \alpha ^\prime v ^\prime _i $.
  \end{enumerate}
Consequently, $ v _i (N) = \phi _i (v) $ is the Shapley value for each player $i \in N $.
  \end{theorem}

For example, calculating $(v_i)_i$ for the glove game from Example \ref{ex:introGlove} yields the following value table.
\be\label{glovegameextended}
\noin
\begin{tabularx}{0.913\textwidth}
{ 
  | >{\centering\arraybackslash}X 
  | >{\centering\arraybackslash}X 
  | >{\centering\arraybackslash}X 
  | >{\centering\arraybackslash}X 
  | >{\centering\arraybackslash}X 
  | >{\centering\arraybackslash}X 
  | >{\centering\arraybackslash}X 
  | >{\centering\arraybackslash}X 
  | }
\hline 
 & $\{1\}$ & $\{2\}$ & $\{3\}$ & $\{1,2\}$ & $\{1,3\}$ & $\{2,3\}$ & $\{1,2,3\}$ \\ [0.5ex]
\hline 
$v_1$ & $\frac{5 }{ 12}$ & $ -\frac{ 5}{ 24}$ & $ -\frac{ 5}{ 24}$ & $\frac{5 }{ 8} $ &$ \frac{5 }{ 8}$ &$ -\frac{ 1}{ 4} $ & $\frac{2 }{ 3} $ \\ [0.5ex]
\hline 
$v_2$ &$-\frac{ 5}{ 24} $ & $\frac{1 }{6 } $ & $  \frac{1 }{ 24} $ & $  \frac{ 3}{ 8}  $ & $0 $ & $\frac{1 }{8} $ & $\frac{1 }{ 6} $ \\ [0.5ex]
\hline
$v_3$ & $-\frac{ 5}{ 24} $ & $ \frac{1 }{ 24} $ & $\frac{1 }{6 }$ & $ 0$ & $ \frac{ 3}{ 8}$ & $\frac{1 }{8}$ & $\frac{1 }{ 6} $ \\ [0.5ex]
\hline
\end{tabularx}
\nn
\ee
The final column indeed corresponds to the Shapley value. The properties in Theorem \ref{thm:gameDecomp} are clearly inspired by the Shapley axioms and serve as a natural extension, successfully recovering the Shapley value at the grand coalition $N$. However, unlike the classic axioms, they are incomplete: they do not uniquely characterize the value $\Phi = \big(\Phi_i(\cdot, \cdot)\big)_{i \in N}$ for all partial coalitions. This motivates our first main result: to give a finite set of properties that fully characterize $\Phi$, which we present in Section \ref{newaxiom}. Our second objective is to determine if a probabilistic interpretation of $\Phi$ exists for any partial coalition $S$, analogous to the Shapley formula. This representation, which involves a random walk on the graph $G$ and a path integration, is presented in Section \ref{newformula}.

\section{Characterization of $\Phi$}\label{newaxiom}

The properties listed in Theorem \ref{thm:gameDecomp} are insufficient to uniquely determine the value function $\Phi$ for all coalitions. Inspired by the long tradition of seeking axiomatic foundations for allocation rules, we introduce a new set of five axioms that provide a complete characterization. Our observation is that a new condition, which we call ``independency," is essential to complement the natural extensions of the classic Shapley axioms. 

To formally state the axioms, we first define some notation. Let $\cal U$ be a countably infinite universe of players, and let $\cN$ be the set of all finite subsets of $\cal U$. The set of all coalition games is $\cG = \bigcup_{N \in \cN} \cG_N$. For players $i,j \in N$ and a coalition $S \subset N$, we define $S^{ij}$ as the coalition obtained by switching the roles of $i$ and $j$ in $S$:
  \begin{equation*}
S^{ij}=
    \begin{cases}
      S & \text{if } \ \{i,j\} \cap S = \varnothing  \, \text{ or } \,  \{i,j\} \subset S, \\
      S \cup \{i\} \setminus \{j\} & \text{if } \ i \notin S \,\text{ and } \, j \in S, \\ 
S \cup \{j\} \setminus \{i\} & \text{if } \ j \notin S \, \text{ and } \, i \in S.
    \end{cases}
  \end{equation*} 
Similarly, for a game $v \in \cG_N$, we define the swapped game $v^{ij} \in \cG_N$ by $v^{ij}(S) = v(S^{ij})$. Intuitively, the contributions of players $i$ and $j$ in game $v$ are interchanged in game $v^{ij}$. Finally, for $i \in N$, let $v_{-i} : 2^{N \setminus \{i\}} \to \R$ be the game $v$ restricted to coalitions not containing $i$, i.e., $v_{-i}(S) = v(S)$ for all $S \subset N \setminus\{i\}$. We now present our five axioms.
 \vspace{1mm}

{\rm {\bf A1} (Efficiency):} $v (S) = \sum_{i \in N} \Phi _i (v,S)$ for any $v \in \cG_N$ and $S \subset N$. \vspace{1mm}

{\rm {\bf A2} (Linearity):} For any $v, v' \in\cG_{N}$, $\al, \al' \in \R$, and $S \subset N$, it holds
\[
\Phi _i (\al v + \al' v', S) = \al \Phi _i (v,S) + \al' \Phi _i (v',S). 
\]
 
{\rm {\bf A3} (Symmetry):} $ \Phi _i (v^{ij}, S^{ij}) =\Phi _j (v, S) $ for all $v \in \cG_N$, $i,j \in N$, and $S \subset N$. \vspace{1mm}
 
Axiom A3 can be interpreted as follows: if the roles of players $i$ and $j$ are interchanged in the game, their payoffs switch accordingly for corresponding coalitions.  \vspace{1mm}

{\rm {\bf A4} (Null-player):} For any $v \in \cG_N$ and $i \in N$, if $\p_i v \equiv 0$, then  
\be
\Phi_j(v, S \cup \{i\}) = \Phi_j(v,S) =\Phi_j(v_{-i},S) \ \text{ for all }  j \in N \setminus \{i\} \text{ and }  S \subset N \setminus \{i\}. \nn
\ee
A4 states that if player $i$ provides no marginal value to any coalition, then the reward for any other player $j$ is independent of player $i$'s participation and is identical to their reward in the game played only by $N\setminus\{i\}$. Combined with A1, this axiom implies that a null player receives nothing: $\Phi_i (v, S)= 0$ for all $S \subset N$. 

Axioms A1–A4 are natural extensions of the classic Shapley axioms and are sufficient to determine the value $\Phi_i (v, N)$ as the Shapley value. However, they do not uniquely determine $\Phi$ for all partial coalitions. The final axiom provides the missing constraint.
\vspace{1mm}

{\rm {\bf A5} (Independency):} For every $v \in \cG_N$ and $i \in N$, the mapping
\be
S \mapsto \frac{\Phi _i (v,S) + \Phi _i (v, S \cup \{ i \})}{2} \, \text{ is constant over all } S \subset N \setminus \{i\}.  \nn
\ee
A5 states that for any player $i$, their average value across a state where they are not in a coalition ($S$) and an adjacent state where they are in it ($S \cup \{i\}$) is independent of the specific coalition $S$. 

We now present our first main theorem. To the best of our knowledge, this is the first result providing a characterization of solutions to Poisson’s equation on graphs.

\begin{theorem}\label{main}
There exists a unique mapping $\Phi = (\Phi_i)_{i \in N} :  \cG_N \times 2^N \to \R^{|N|}$ that satisfies axioms {\bf A1–A5} for all $N \in \cN$ and the initial condition $\Phi(v, \varnothing) = {\bf 0}$. This unique map $\Phi$ is the solution to the Poisson's equation \eqref{STPoisson}. 
\end{theorem}

\section{Probabilistic representation of $\Phi$ via diffusion path integration}\label{newformula}

The classic Shapley formula \eqref{eqn:shapleyPermutation} assumes a coalition formation process that is strictly increasing, with players joining one by one in a uniformly random order. We now consider a more general process by defining a diffusion (random walk) on the coalition space $\cV$. Let $(X_t)_{t \in \N_0}$ be a Markov chain on the state space $\cV$ starting at $X_0 = \varnothing$, with transition probabilities given by:
 \begin{align}\label{MC}
p_{S,T} := 1/|N|
\ \text{ if } \ T \sim S, \q  p_{S,T} := 0 \ \text{ if } \ T \not\sim S. 
\end{align} 
This transition law can be interpreted as a baseline model where, at any point in time, each player is equally likely to change their status by either joining or leaving the current coalition. Critically, this process allows players not only to join but also to leave, meaning transitions from $S$ to $S \setminus \{i\}$ are possible.

Let $(\Omega, \cF, \cP)$ be the underlying probability space. For each target coalition $S \in \cV$ and each sample path $\omega \in \Omega$, let $\tau_{S} = \tau_{S}(\omega)$ be the first hitting time of state $S$, i.e., the first time $t \ge 1$ such that $X_{t}(\omega) = S$.
\begin{figure}[h]
\centering
\begin{subfigure}{.5\textwidth}
  \centering
  \includegraphics[width=.80\linewidth]{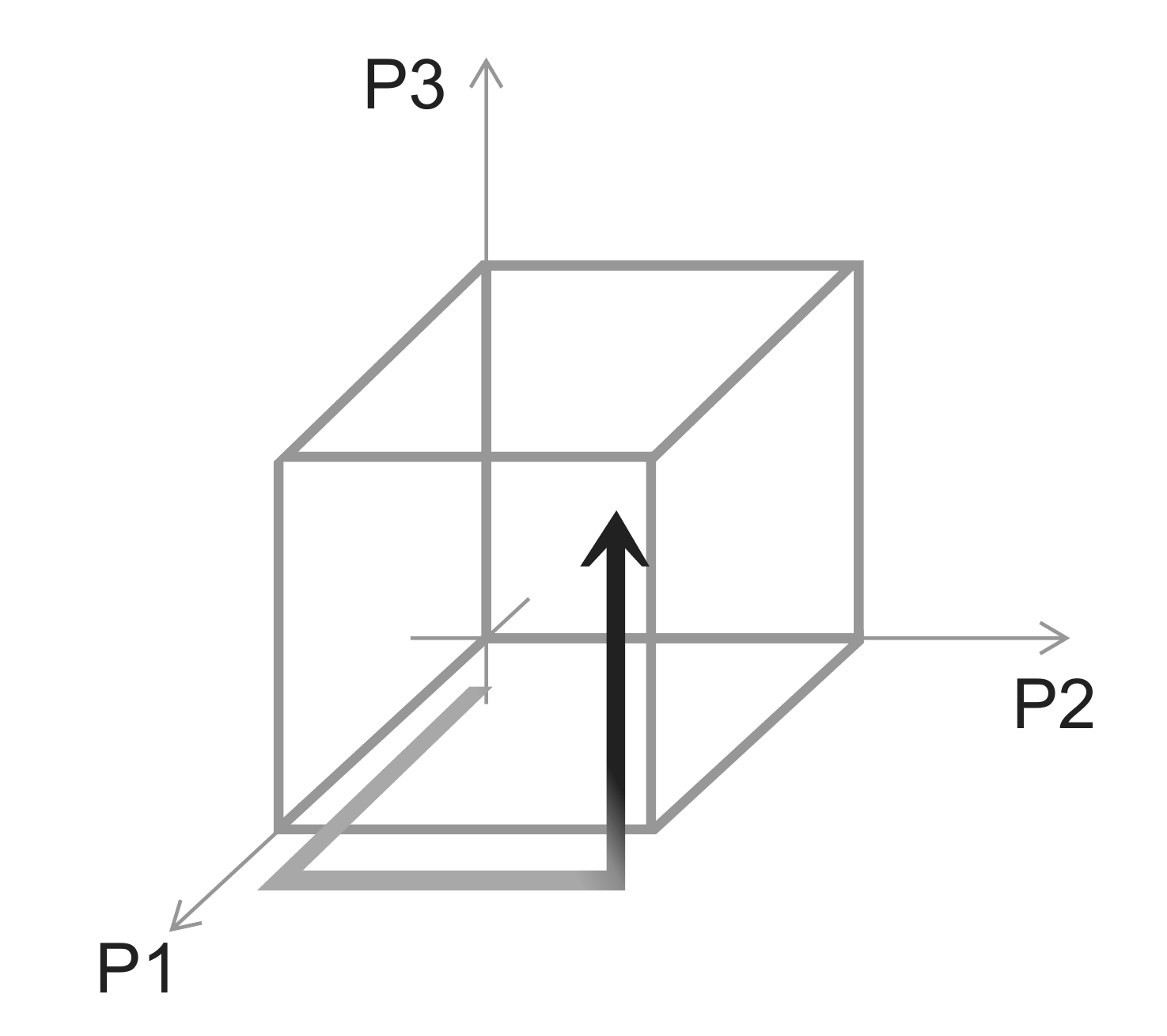}
\end{subfigure}%
\begin{subfigure}{.5\textwidth}
  \centering
  \includegraphics[width=.80\linewidth]{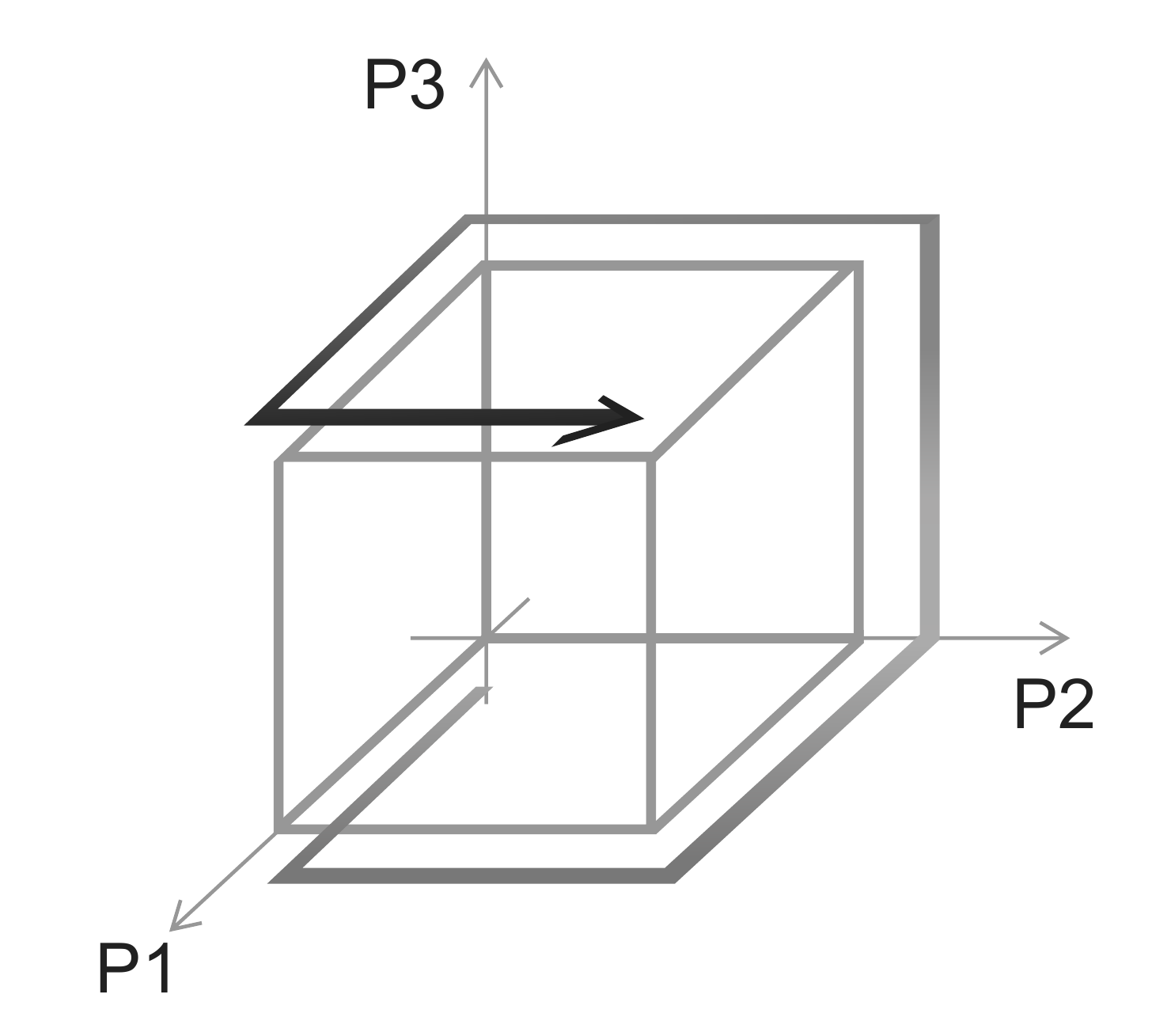}
\end{subfigure}
\caption{An increasing coalition path as in the Shapley formula (left) and a more general coalition path from our random walk model (right).}
\label{figure2}
\end{figure}

Given a game $v \in \cG_N$, we can calculate the total marginal contribution of player $i$ along a sample path $\omega$ that travels from $\varnothing$ to $S$ using the following path integral:
\be\label{pathintegral}
{\cal I}_i(v, S) = {\cal I}_i(v, S)(\omega) := \sum_{t=1}^{\tau_{S}(\omega)} \p_i  v \big(X_{t-1}(\omega), X_t(\omega) \big).
\ee
Here, $\p_i v$ is the marginal contribution of player $i$ as defined in \eqref{partialdifferential}, capturing gains from joining and losses from leaving. Thus, ${\cal I}_i(v, S)$ represents player $i$'s net contribution accumulated along the specific path $\omega$ until the coalition first reaches state $S$. By averaging this quantity over all possible paths, we define a new value function:
\be\label{value}
\Psi_i (v, S) := \E[ {\cal I}_i(v, S)] = \int_\Omega  {\cal I}_i(v, S)(\omega)\, d \cP(\omega).
\ee
This value, $\Psi_i(v, S)$, represents player $i$'s expected total contribution, given that the coalition formation process begins at $\varnothing$ and eventually reaches the state $S$. Our second main result states that this probabilistically defined value is identical to the axiomatically characterized value from the previous section.

\begin{theorem}\label{coincide}
$\Phi = \Psi$. That is, $\Phi_i(v,S) = \Psi_i(v,S)$ for all $v \in \cG_N$, $S \subset N$, and $i \in N$. In particular, the classic Shapley value is recovered for the grand coalition: $\phi_i(v) = \Psi_i(v,N)$ for every $v \in \cG_N$ and $i \in N$. 
\end{theorem}

Theorem \ref{coincide} provides a novel interpretation of the Shapley value: $\phi_i(v)$ is agent $i$'s expected total contribution when a random coalition process, allowing both entry and exit, terminates upon reaching the grand coalition $N$. Even in this special case, the underlying formulas \eqref{eqn:shapleyPermutation} and \eqref{value} are structurally different. The classic Shapley formula is a finite sum over $|N|!$ deterministic, strictly increasing paths, whereas \eqref{value} is an infinite sum over all possible random paths, including those that are non-monotonic. 

\begin{example}\label{twoperson}
We demonstrate a direct calculation of the value $\Psi$ for a general two-person game. Let $v_1 = v(\{1\})$, $v_2 = v(\{2\})$, and $v_{12} = v(\{1,2\})$. The classic Shapley formula \eqref{eqn:shapleyPermutation} yields
\be
\phi_1(v) = \tfrac{1}{2}(v_1 + v_{12} - v_2), \q \phi_2(v) = \tfrac{1}{2}(v_2 + v_{12} - v_1). \nn 
\ee
To calculate $\Psi_1(v, N)$ where $N=\{1,2\}$, we consider all paths from $\varnothing$ to $N$. Due to the sign-changing property of $\p_i$ (i.e., $\p_i v \bigl(B, A \bigr) = - \p_i v \bigl( A, B \bigr)$), many terms in the path integral cancel. For example, any path segment like $(\dots, S, S \cup \{i\}, S, \dots)$ contributes nothing to player $i$'s total. It turns out that for any path $\omega$ from $\varnothing$ to $N$, the path integral $\cI_1(\omega)$ is either $v_1$ (if the path hits $\{1,2\}$ after $\{1\}$) or $v_{12}-v_2$ (if it hits $\{1,2\}$ after $\{2\}$). The probabilities of these two events are equal. Thus,
\[
\Psi_1(v, N) = \tfrac{1}{2}(v_1) + \tfrac{1}{2}(v_{12}-v_2) = \phi_1(v).
\]
A similar argument shows $\Psi_2(v, N)=\phi_2(v)$. 

Now, let the terminal coalition be $\{1\}$. We compute $\Psi_1(v, \{1\})$ and $\Psi_2(v, \{1\})$.
\be
\begin{tabularx}{0.9\textwidth}{bss}
\hline
Path $\omega$ from $\varnothing$ to $\{1\}$ & $\cI_1(\omega)$ & $\cI_2(\omega)$ \\ \hline
$(\varnothing, \{1\})$ & $v_1$ & $0$ \\ 
$(\varnothing, \{2\}, \varnothing, \{1\})$ & $v_1$ & $0$ \\ 
$(\varnothing, \{2\}, \{1,2\}, \{1\})$ & $v_{12}-v_2$ & $v_2 - (v_{12}-v_1)$ \\ 
$(\varnothing, \{2\}, \varnothing, \{2\}, \varnothing, \{1\})$ & $v_1$ & $0$ \\ 
$\dots$ & $\dots$ & $\dots$ \\ \hline
\end{tabularx}
\nn
\ee
The transition probability at each step is $1/2$ under the law \eqref{MC}. Summing the contributions weighted by their path probabilities (= $(\frac12)^{\text{(\# of steps)}}$) gives:
\begin{align*}
\Psi_1(v, \{1\}) &= \tfrac{1}{2} v_1 + (\tfrac{1}{2})^3 (v_1 + v_{12} - v_2) + (\tfrac{1}{2})^5 \cdot 2 (v_1 + v_{12} - v_2) + \dots \\
&= \tfrac{1}{2} v_1 + \left( \sum_{k=1}^\infty (\tfrac{1}{2})^{2k+1} 2^{k-1} \right) (v_1 + v_{12} - v_2) \\
&= \tfrac{1}{2} v_1 + \tfrac{1}{4} (v_1 + v_{12} - v_2) = \tfrac{1}{4}(3v_1 - v_2 + v_{12}). 
\end{align*}
The complete value table for the two-person game is:
\be\label{twopersontable}
\centering
\begin{tabularx}{0.9\textwidth}
{ 
  | >{\centering\arraybackslash}X
  | >{\centering\arraybackslash}X 
  | >{\centering\arraybackslash}X 
  | >{\centering\arraybackslash}X 
  | }
\hline 
 & $\{1\}$ & $\{2\}$ & $\{1,2\}$ \\ [0.5ex]
\hline 
$\Psi_1$ & $\frac14(3v_1 -v_2 + v_{12})$  & $\frac14(v_1+v_2-v_{12})$ & $\frac12(v_1-v_2+v_{12})$  \\ [0.5ex]
\hline 
$\Psi_2$ & $\frac14(v_1+v_2-v_{12})$  & $\frac14(3v_2 -v_1 + v_{12})$ & $\frac12(v_2-v_1+v_{12})$  \\ [0.5ex]
\hline
\end{tabularx}
\nn
\ee
\end{example}

We recall that in the glove game, some negative values were assigned. For instance, $v_1(\{2,3\}) = -\tfrac{1}{4}$, whereas $v_2(\{2,3\}) = v_3(\{2,3\}) = \tfrac{1}{8}$. This phenomenon can be explained as follows. The coalition $\{2,3\}$ has three possible immediate predecessors: $\{2\}$, $\{3\}$, and $\{1,2,3\}$. Among these, the transition $\{1,2,3\} \to \{2,3\}$---corresponding to player 1 leaving the grand coalition---induces a drop in the coalition value from $v(\{1,2,3\}) = 1$ to $v(\{2,3\}) = 0$. This decline is represented by $\partial_1 v\big((\{1,2,3\}, \{2,3\})\big) = -1$,
and the resulting loss is partially attributed to player 1, yielding her assigned value of $-\tfrac{1}{4}$.

While direct calculation of $\Psi$ using path enumeration is feasible for small games, it becomes challenging as $|N|$ increases. Theorem \ref{coincide} establishes a crucial link: the value $\Psi$, defined by an infinite sum over random paths, can be efficiently computed by solving the system of linear equations \eqref{STPoisson} that defines $\Phi$. This duality provides both a practical computation method and a deep probabilistic meaning for this extended Shapley value based on the combinatorial Hodge theory.

\section{Conclusion}\label{conclusion}
In this paper, we have developed a complete theoretical foundation for the Hodge-theoretic cooperative game value proposed by Stern and Tettenhorst. Their innovative framework extended the concept of player value to all partial coalitions but lacked a full axiomatic characterization and a probabilistic interpretation, two cornerstones of the classic Shapley value's broad acceptance.

Our work resolves these two critical gaps. First, we introduced a new set of five axioms—efficiency, linearity, symmetry, an extended null-player condition, and a novel independency axiom—and proved in Theorem \ref{main} that they uniquely characterize the value solution for any game and any coalition. This axiomatic framework provides a theoretical foundation for this value as a principle of allocation across the entire coalition lattice, not just for the grand coalition.

Second, we formulated a new probabilistic representation for the value in Theorem \ref{coincide}. We showed that the value of a player in any given coalition is equal to their expected total marginal contribution accumulated along a random walk on the coalition graph. This path integral formulation generalizes the classic Shapley formula from a model of deterministic, ordered coalition growth to a more flexible and dynamic process where players can both join and leave coalitions. This result provides a clear, intuitive meaning to the value assigned at each partial coalition state.

By providing both a unique axiomatic basis and a compelling probabilistic interpretation, our findings establish the Hodge-theoretic value as a robust and canonical extension of the Shapley value. This dual perspective solidifies its standing as a principled method for allocating value in cooperative settings where the formation of the grand coalition is not guaranteed. Future research could focus on developing efficient computational algorithms for larger games and applying this comprehensive framework to practical problems in economics, political science, and machine learning, where assessing contributions within partial alliances is crucial.

\section{Proof of Theorem \ref{main} and Theorem \ref{coincide}}\label{proofs}
We introduce basic linear spaces and operators in combinatorial Hodge theory.  Let $\ell^2(\cV)$ denote the space of functions  $ \cV \rightarrow \mathbb{R} $ with the inner product
\begin{equation}\label{inner1}
   \langle u , v \rangle \coloneqq \sum _{ S \in \cV }  u (S) v (S).
\end{equation}
We recall that these are called coalition games if $\cV = 2^N$ and $v(\varnothing) = 0$. In contrast, $\cV$ can now be an arbitrary finite set and $v \in \ell^2(\cV)$ is not required to assume $0$ anywhere. 

Let $ \ell ^2 (\cE) $ denote the space of functions
$ \overline{\cE} \rightarrow \mathbb{R} $ equipped with the inner product
\begin{equation}\label{inner2}
   \langle f , g \rangle \coloneqq \sum _{ (S,T) \in \cE } f (S,T) g (S,T)
\end{equation}
with the sign changing property $ f(T,S) = - f (S,T) $. Observe that $\p_i v \in \ell ^2 (\cE) $. 

Now the gradient operator $ \mathrm{d} \colon \ell ^2 (\cV) \rightarrow \ell ^2 (\cE) $ is defined by
\begin{equation}\label{gradient}
  \mathrm{d} v ( S,T ) \coloneqq v (T) - v (S) \ \text{ for each } (S,T) \in \overline{\cE}.
\end{equation}
For a function $v$ on $\cV$, $ \mathrm{d} v$ measures its marginal value for each edge $(S,T) \in \overline\cE$. 

Let $ \mathrm{d}^* \colon \ell ^2 (\cE) \rightarrow \ell^2 (\cV) $ denote the adjoint of $ \mathrm{d}$. $ \mathrm{d}^*$ is called the divergence operator, which is characterized by the defining relation for the adjoint operator
\be\label{adjointdefinition}
   \langle \mathrm{d}v, f \rangle_{ \ell ^2 (\cE)} =    \langle v, \mathrm{d}^*f \rangle_{ \ell ^2 (\cV)} \ \text{ for every } v \in \ell^2(\cV) \text{ and } f \in \ell^2(\cE).
\ee
The adjoint property yields the following explicit form of the divergence
\begin{align}\label{div}
\mathrm{d}^*f (S) 
=  \sum_{T \sim S} f(T,S) \ \text{ for every } S \in \cV,
\end{align}
where $ T \sim S $ indicates $S$ and $T$ are adjacent by an edge in the graph $G$. 

The Laplacian is the symmetric (self-adjoint) operator $ \mathrm{L} = \mathrm{d} ^\ast \mathrm{d} : \ell^2(\cV) \to \ell^2(\cV)$:
\begin{align*}
\mathrm{L} v (S) =  \sum_{T \sim S}\big( v(S) - v(T) \big).
\end{align*}
$\mathrm{L} v(S)$ calculates the sum of $v$'s marginal increment directed to each state $S$. 

With these and the notation $v_i (S) = \Phi_i(v, S)$ for all $S \subset N$, the balance equation \eqref{STPoisson} can now be written as the following Poisson's equation
\be\label{ls1}
\mathrm{L} v_i = \mathrm{d}^* \p_i v \ \text{ with } v_i(\es) = 0, \ \text{ for each } i \in N.
\ee

\begin{proof}[Proof of Theorem \ref{main}] 
First, we claim that A1--A5 determines the operator $\Phi$ uniquely (if exists). For each player set $N$, define games $\delta_{S,N} \in \cG_N$ for each $ \es \neq S \subset N$ by
\[
\delta_{S,N}(S) = 1, \q 
      \delta_{S,N}(T)= 0 \,\text{ if } \, T \neq S.
\]
We proceed by an induction on $|N|$. The case $|N|=1$ is already from A1. Suppose the claim holds for $|N|-1$, so $\Phi_j(\delta_{S, N \setminus \{i\}}, \cdot)$ are determined for all $i,j \in N$ and $\es \neq S \subset N \setminus \{i\}$. Define the games $\Delta_{(S, S \cup \{i\})} \in \cG_N$ for each $\es \neq S \subset N \setminus \{i\}$ by 
\[
\Delta_{(S, S \cup \{i\})}(T) = 1 \, \text{ if } \, T=S \text{ or } T=S \cup \{i\}, \q \Delta_{(S, S \cup \{i\})}(T) =0 \, \text{ otherwise.}
\]
Notice then A4 (and induction hypothesis) determines $\Phi$ for all $\Delta_{(S, S \cup \{i\})} \in \cG_N$. Then thanks to A2, to prove the claim, it is enough to show that A1--A5 can determine $\Phi$ for the pure bargaining game $\delta:= \delta_{N,N}$, because for any $\es \neq S \subset N$, we can write $\delta_{S,N}$ as the following sign-alternating sum
\be
\delta_{S,N} = \Delta_{(S, S \cup \{i_1\})} - \Delta_{(S \cup \{i_1\}, S \cup \{i_1, i_2\})} + \Delta_{(S \cup \{i_1, i_2\}, S \cup \{i_1, i_2, i_3\})} - \dots \pm \delta_{N,N}. \nn
\ee
By A3, $\sum_{S \subset N} \Phi_i(\delta,S)$ is constant for all $i \in N$, thus equals $ 1/|N|$ by A1. Define 
\[
u_i (S) := \Phi_i(\delta,S) - \frac{1}{|N| 2^{|N|}} \ \text{ for all } \ S \subset N
\]
so that $u_i (\varnothing) = -\frac{1}{|N| 2^{|N|}}$ and $\sum_{S \subset N} u_i(S)=0$ for all $i$. Now observe A5 implies:
\be
u_i(S) + u_i(S \cup \{i\}) \text{ is constant for all } S \subset N \setminus \{i\}, \text{ hence it is zero.} \nn
\ee
This determines $u_i$, thus $\Phi_i(\delta, \cdot)$, as follows: suppose $u_i(S)$ has been determined for all $i$ and $|S| \le k-1$. Let $|T| =k \le |N|-1$. Then we have $u_i(T) = - u_i (T \setminus \{i\})$ for all $i \in T$ and it is constant (say $c_k$) by A3. Using A1 and A3, we obtain
\be
0 = \delta(T)=\sum_{i\in N}\Phi_i(\delta, T)=\sum_{i\in N}\bigg(u_i(T) + \frac{1}{|N| 2^{|N|}} \bigg), \nn
\ee
yielding $\sum_{i\in N} u_i(T) = -1/ 2^{|N|}$. With $u_i(T) = c_k$ for all $i \in T$, we deduce $u_j(T) =  \frac{ -1 - 2^{|N|} k c_k } {2^{|N|}(|N|-k) }$ for all $j \notin T$. This shows $u_i(T)$ is determined for all $|T|=k \le |N|-1$. Of course, $\Phi_i(\delta,N) = 1/|N|$ for all $i \in N$ by A1 and A3. By induction (on $|N|$ and on $k$ for each $N$), the proof of uniqueness of the operator $\Phi$ is therefore complete.

It remains to show that the component games $(v_i)_{i \in N}$ solving \eqref{ls1}  satisfies A1--A5. Firstly, A2 is clearly satisfied. To show that A1 is satisfied, we compute 
  \begin{equation*}
    \mathrm{L} \sum _{ i \in N } v _i = \sum _{ i \in N}  \mathrm{L} v _i = \sum _{ i \in N }   \mathrm{d}^\ast \p_i v =   \mathrm{d}^\ast \sum _{ i \in N } \p_i v =   \mathrm{d}^\ast \mathrm{d} v = \mathrm{L}v,
  \end{equation*}
since $ \mathrm{d} = \sum _{ i \in N } \p_i $. Hence by unique solvability of \eqref{ls1}, $\sum _{ i \in N } v _i =v$ as desired. 

Next, let $\sigma$ be a permutation of $N$. Let $\sigma$ act on $\ell ^2 (2^{N})$ and $\ell ^2 (\cE)$ via
\be
\sigma v(S) = v(\sigma (S)) \ \text{and} \ \sigma f\bigl( S , S \cup \{ i \}\bigr) =  f\bigl( \sigma (S) , \sigma (S \cup \{ i \})\bigr),  \ v \in \ell ^2 (2^{N}), \ f \in \ell ^2 (\cE).  \nn
\ee
It is easy to check $ \mathrm{d} \sigma = \sigma \mathrm{d} $ and
  $ \mathrm{d} _i \sigma  = \sigma \mathrm{d} _{ \sigma (i)
  } $. We also have $\mathrm{d}^\ast \sigma = \sigma \mathrm{d}^\ast$, since 
\be
  \langle v , \mathrm{d}^\ast \sigma f \rangle =    \langle \mathrm{d} v ,  \sigma f \rangle =  \langle \sigma^{-1} \mathrm{d} v ,   f \rangle =    \langle  \mathrm{d}  \sigma^{-1}  v ,   f \rangle =  \langle   \sigma^{-1}  v ,  \mathrm{d}^\ast f \rangle=   \langle   v ,  \sigma  \mathrm{d}^\ast f \rangle \nn
\ee
for any $v \in \ell ^2 (2^{N})$, $f \in \ell ^2 (\cE)$. Now let $\sigma$ be the transposition of $i,j$. We have
\be
   \mathrm{L}(\sigma v)_i = \mathrm{d}^\ast \p_i \sigma v = \mathrm{d}^\ast  \sigma  \p_j v =   \sigma \mathrm{d}^\ast \p_j v = \sigma  \mathrm{L} v_j =  \mathrm{L}\sigma v_j \nn
\ee
which shows $(\sigma v)_i = \sigma v_j$ by the unique solvability. Notice this corresponds to A3.
  
For A4, let $v \in \cG_N$, $i \in N$, and assume $\p_i v = 0$. Then from \eqref{ls1} we readily get $v_i \equiv 0$. Fix $j \neq i$, and let $\mathrm{\tilde d}$, ${\tilde \p}_j$ be the differential operators restricted on $2^{N \setminus \{i\}}$, and set $\tilde v = v_{-i}$, i.e., $\tilde v$ is the restriction of $v$ on $2^{N \setminus \{i\}}$. Let $\tilde v_j : 2^{N \setminus \{i\}} \to \R$ be the solution to the equation $\mathrm{\tilde d}^\ast \mathrm{\tilde d} \tilde v_j = \mathrm{\tilde d}^\ast {\tilde \p}_j \tilde v$ with $\tilde v_j (\es) = 0$. Finally, in view of A4, define $v_j \in \cG_N$ by $v_j = \tilde v_j$ on $2^{N \setminus \{i\} }$ and $\p_i v_j = 0$. Now observe that A4 will follow if we can verify that this $v_j$ indeed solves the equation $\mathrm{ d}^\ast \mathrm{d}  v_j = \mathrm{ d}^\ast \p_j  v$.

To show this, let $S \subset N \setminus \{i\}$. In fact the following string of equalities holds:
\be
\mathrm{d}^\ast \mathrm{d} v_j (S \cup \{i\}) 
=\mathrm{d}^\ast \mathrm{d} v_j (S) 
=\mathrm{\tilde d}^\ast \mathrm{\tilde d} \tilde v_j (S) 
=\mathrm{\tilde d}^\ast {\tilde \p}_j \tilde v (S) 
=\mathrm{d}^\ast \p_j  v (S) 
=\mathrm{d}^\ast \p_j  v (S \cup \{i\})
\nn
\ee
which simply follows from the definition of the differential operators. For instance
\begin{align*}
\mathrm{d}^\ast \mathrm{d} v_j (S)   
= \sum _{ T \sim S } \mathrm{d} v_j ( T, S ) 
= \sum _{ T \sim S, \, T \neq S \cup \{i\} } \mathrm{d} v_j ( T, S )
=\mathrm{\tilde d}^\ast \mathrm{\tilde d} \tilde v_j (S)
\end{align*}
where the second equality is due to $\p_i v_j = 0$. On the other hand, since $j \neq i$,
\begin{align*}
\mathrm{d}^\ast \p_j v (S)   
= \sum _{ T \sim S } \p_j v ( T, S ) 
=  \sum _{ T \sim S } {\tilde \p}_j \tilde v ( T, S ) 
=\mathrm{\tilde d}^\ast {\tilde \p}_j \tilde v (S).
\end{align*}
The first equality is due to the definition of $v_j$ (i.e. $v_j = \tilde v_j$ on $2^{N \setminus \{i\} }$ and $\p_i v_j = 0$), and the last equality is due to $\p_i v= 0$. This verifies A4.

Finally, we verify A5. For this, we need to verify the following claim:
\be\label{constancycondition}
v_i (S) + v_i (S \cup \{i\}) \, \text{ is constant over all } \, S \subset N \setminus \{i\}. 
\ee
Let $S \subset N \setminus \{i\}$, and recall $\mathrm{d}^\ast \p_i v (S) = v(S) - v(S \cup \{i\}) = -\mathrm{d}^\ast \p_i v (S \cup \{i\})$. Hence, $\mathrm{L} v_i (S) +\mathrm{L} v_i (S \cup \{i\} ) = 0$. Define $w_i \in \ell^2(2^{N})$ by $w_i (S) = v_i(S \cup \{i\} )$ and $w_i (S \cup \{i\} )= v_i(S  )$ for all $S \subset N \setminus \{i\}$. Then clearly $ \mathrm{L} v_i (S \cup \{i\} ) = \mathrm{L} w_i (S )$ and $\mathrm{L} v_i (S) = \mathrm{L} w_i (S \cup \{i\}  )$. Thus $\mathrm{L} (v_i + w_i) \equiv 0$, hence $v_i + w_i \in  \mathcal{N} ( \mathrm{d} )$, meaning that $v_i + w_i$ is constant. 
\end{proof}


\begin{proof}[Proof of Theorem \ref{coincide}]
Fix $v \in \cG_N$. Thanks to Theorem \ref{main}, it is enough to show that $\Psi_i(S):= \Psi_i (v, S)$ satisfies A1--A5. The linearity A2 trivially holds. For A1, observe
\be
\sum_{i \in N} {\cal I}_i(v, S) = \sum_{i \in N} \sum_{m=1}^{\tau_S} \p_i v \big(X_{m-1}, X_m \big) =  \sum_{m=1}^{\tau_S} \mathrm{d} v \big(X_{m-1}, X_m \big) = v(S) - v(\varnothing) = v(S) \nn
\ee
since $\mathrm{d} = \sum_i \p_i$. Thus $\sum_i \Psi_i(S) = \sum_i \E[{\cal I}_i(v, S)] = \E[\sum_i {\cal I}_i(v, S)] = v(S)$, showing A1.

For any $i \neq j$, note that each sample path $\omega$ has its counterpart $\omega^{ij}$, defined by
\be
X_m(\omega) = S \ \text{ if and only if } \ X_m(\omega^{ij}) = S^{ij} \ \text{ for every $m \in \N$ and } S \subset N.  \nn
\ee
With this, observe ${\cal I}_{i}(v^{ij},S^{ij})(\omega)  = {\cal I}_{j}(v, S)(\omega^{ij})$. Taking expectation then verifies A3.
 
 For A4, assume $\p_i v= 0$. Then ${\cal I}_i(v, \cdot) = 0$ readily gives $\Psi_i \equiv 0$. For $j \in N \setminus \{i\}$, let $(\tilde X_m)_m$ denote the random walk on the restricted state space $2^{N \setminus \{i\}}$. 
Observe that $(\tilde X_m)_m$ can be embedded in $( X_m)_m$ via the identification
 \be
 \tilde X_m := S \in 2^{N \setminus \{i\}}  \  \text{ if and only if } \ X_m = S  \, \text{ or } \, X_m = S \cup  \{i\}, \nn
 \ee
which implies the following identity
\be\label{111}
\Psi_j (v_{-i}, S) = \int_\Omega \sum_{m=1}^{\tau_S \wedge \tau_{S \cup  \{i\}}} \p_j v (X_{m-1}, X_m) d\cP \q \text{for every } j \in N \setminus \{i\}.
\ee
Consider an arbitrary connected finite path $\tta: X_0 \to X_1 \to \dots \to X_m$ on the hypercube graph  
where $X_0 = S$ and $X_m = S \cup \{i\}$, or $X_0 = S \cup \{i\}$ and $X_m = S$. Its reversed path is then given by 
$\tta': X_0' \to  \dots \to X_m'$ where $X_k' = X_{m-k}$. Then flip the reversed path with respect to $i$ and get the path $\tta^*: X_0^* \to  \dots \to X_m^*$, where $X^*_k = S \in 2^{N \setminus \{i\}}$ if $X'_k = S \cup \{i\}$  and $X^*_k = S \cup \{i\}$ if $X'_k = S$. Observe that the correspondence $\tta \leftrightarrow \tta^*$ is bijective, both either starts at $S$ and ends at $S \cup \{i\}$ or $S\cup \{i\}$ and $S$, and we have
\be
 \sum_{k=1}^{m} \p_j v (X_{k-1}, X_k)
 =   - \sum_{k=1}^{m} \p_j v (X'_{k-1}, X'_k)
 =  - \sum_{k=1}^{m} \p_j v (X^*_{k-1}, X^*_k) \ \text{ for all } j \in N \setminus \{i\}, \nn
\ee
where the second identity is due to $\p_i v \equiv 0$. The sign reversing implies there is zero expected gain by integrating $\p_j v$ from $S$ to $S \cup \{i\}$ or conversely  through random walk. The equation \eqref{111} therefore implies A4 holds: $\Psi_j(v_{-i}, S) = \Psi_j(v, S) = \Psi_j(v, S \cup  \{i\})$.


 Finally, let us verify A5. For any states $S,T \subset N$, consider the random walk $(X^S_m)_m$ whose initial state is $X^S_0 = S$ (instead of $\varnothing$). Define analogously ${\cal I}^S_i (v, T)$ and 
 $\Psi^S_i (v, T) = \E[{\cal I}^S_i (v, T)]$ using $X^S$ in place of $X$. Write $\Psi^S_i (v, T) = \Psi^S_i (T)$ for simplicity. By the correspondence between paths from $S$ to $T$ and their reverse from $T$ to $S$, it readily follows $\Psi^S_i (T) =  - \Psi^T_i (S)$. Now to verify A5 for $\Psi$, we need two identities. The first is:
 \be\label{identity}
\Psi_i (T)  - \Psi_i ( S) = \Psi^S_i ( T),
 \ee
which follows from the Markov property of the random walk. To see this, we compute 
\begin{align*}
{\cal I}_{i} (v, T) - {\cal I}_{i} (v, S) &= \sum_{m=1}^{\tau_T} \p_i v \big( X_{m-1}, X_m \big) 
-  \sum_{m=1}^{\tau_S} \p_i v \big( X_{m-1}, X_m \big) \\
&= {\bf 1}_{\tau_S < \tau_T}\sum_{m=\tau_S + 1}^{\tau_T} \p_i v \big( X_{m-1}, X_m \big) 
- {\bf 1}_{\tau_T < \tau_S}\sum_{m = \tau_T + 1}^{\tau_S} \p_i v \big( X_{m-1}, X_m \big).
\end{align*}
By taking expectation, we obtain via the Markov property
\begin{align*}
\E[{\cal I}_{i} (v, T)]  - \E[{\cal I}_{i} (v, S)]
 = \cP(\{\tau_S < \tau_T\}) \Psi^S_i (T) 
- \cP(\{\tau_T < \tau_S\}) \Psi^T_i (S) = \Psi^S_i (T),
\end{align*}
which shows \eqref{identity}. The second identity is:
\be\label{222}
\Psi^{S \cup \{i\}}_i( T\cup \{i\})
 = - \Psi^{S}_i( T) \q \text{for any } S, T \subset N \setminus \{i\}.
\ee
To see this, Let $S,T \subset N \setminus \{i\}$, and consider an arbitrary coalition path 
$\omega: X_0 \to X_1 \to \dots \to X_m$ 
where $X_0 = S$, $X_m = T$, and $(X_t, X_{t+1}) \in \overline{\cE}$. The flip of $\omega$ with respect to $i$ is then
$\omega^*: X^*_0 \to X^*_1 \to \dots \to X^*_m$ 
where $X^*_t := X_t \cup \{i\}$ if $i \notin X_t$, and $X^*_t := X_t \setminus \{i\}$ if $i \in X_t$. Then similarly as before, we obtain
\be
 \sum_{k=1}^{m} \p_i v (X_{k-1}, X_k)
 =  - \sum_{k=1}^{m} \p_i v (X^*_{k-1}, X^*_k), \nn
\ee
since only those transitions including or excluding $i$ yield nonzero $\p_i v$. Taking expectation over all random path $\omega$ implies \eqref{222}. With \eqref{identity}, for any $S, T \subset N \setminus \{i\}$, we deduce
\begin{align*}
 \Psi_i( T\cup \{i\}) - \Psi_i ( S \cup \{i\}) = \Psi^{S \cup \{i\}}_i( T\cup \{i\})
 = - \Psi^{S}_i( T) = - \big( \Psi_i(T) - \Psi_i (S) \big).
\end{align*}
This verifies A5, and hence by Theorem \ref{main}, completes the proof.
\end{proof}
\newpage

\bibliographystyle{plainnat}
\bibliography{TLbib}

\begin{thebibliography}{13}
\providecommand{\natexlab}[1]{#1}
\providecommand{\url}[1]{\texttt{#1}}
\expandafter\ifx\csname urlstyle\endcsname\relax
  \providecommand{\doi}[1]{doi: #1}\else
  \providecommand{\doi}{doi: \begingroup \urlstyle{rm}\Url}\fi

\bibitem[Algaba et~al.(2019)Algaba, Fragnelli, and
  S{\'a}nchez-Soriano]{algaba2019handbook}
Encarnaci{\'o}n Algaba, Vito Fragnelli, and Joaqu{\'\i}n S{\'a}nchez-Soriano.
\newblock \emph{Handbook of the Shapley value}.
\newblock CRC Press, 2019.

\bibitem[Candogan et~al.(2011)Candogan, Menache, Ozdaglar, and
  Parrilo]{candogan2011flows}
Ozan Candogan, Ishai Menache, Asuman Ozdaglar, and Pablo~A Parrilo.
\newblock Flows and decompositions of games: Harmonic and potential games.
\newblock \emph{Mathematics of Operations Research}, 36\penalty0 (3):\penalty0
  474--503, 2011.

\bibitem[Ghorbani and Zou(2019)]{ghorbani2019data}
Amirata Ghorbani and James Zou.
\newblock Data shapley: Equitable valuation of data for machine learning.
\newblock In \emph{International conference on machine learning}, pages
  2242--2251. PMLR, 2019.

\bibitem[Jiang et~al.(2011)Jiang, Lim, Yao, and Ye]{jiang2011statistical}
Xiaoye Jiang, Lek-Heng Lim, Yuan Yao, and Yinyu Ye.
\newblock Statistical ranking and combinatorial hodge theory.
\newblock \emph{Mathematical Programming}, 127\penalty0 (1):\penalty0 203--244,
  2011.

\bibitem[Lim(2020)]{lim2020hodge}
Lek-Heng Lim.
\newblock Hodge laplacians on graphs.
\newblock \emph{SIAM Review}, 62\penalty0 (3):\penalty0 685--715, 2020.

\bibitem[Mitchell et~al.(2022)Mitchell, Cooper, Frank, and
  Holmes]{mitchell2022sampling}
Rory Mitchell, Joshua Cooper, Eibe Frank, and Geoffrey Holmes.
\newblock Sampling permutations for shapley value estimation.
\newblock \emph{Journal of Machine Learning Research}, 23\penalty0
  (43):\penalty0 1--46, 2022.

\bibitem[Pang et~al.(2021)Pang, Yu, and Liu]{pang2021correlation}
Chuanjun Pang, Jianming Yu, and Yan Liu.
\newblock Correlation analysis of factors affecting wind power based on machine
  learning and shapley value.
\newblock \emph{IET Energy Systems Integration}, 3\penalty0 (3):\penalty0
  227--237, 2021.

\bibitem[Rodr{\'\i}guez-P{\'e}rez and
  Bajorath(2019)]{rodriguez2019interpretation}
Raquel Rodr{\'\i}guez-P{\'e}rez and J{\"u}rgen Bajorath.
\newblock Interpretation of compound activity predictions from complex machine
  learning models using local approximations and shapley values.
\newblock \emph{Journal of medicinal chemistry}, 63\penalty0 (16):\penalty0
  8761--8777, 2019.

\bibitem[Rozemberczki et~al.(2022)Rozemberczki, Watson, Bayer, Yang, Kiss,
  Nilsson, and Sarkar]{rozemberczki2022shapley}
Benedek Rozemberczki, Lauren Watson, P{\'e}ter Bayer, Hao-Tsung Yang,
  Oliv{\'e}r Kiss, Sebastian Nilsson, and Rik Sarkar.
\newblock The shapley value in machine learning.
\newblock \emph{arXiv preprint arXiv:2202.05594}, 2022.

\bibitem[Schoch et~al.(2022)Schoch, Xu, and Ji]{schoch2022cs}
Stephanie Schoch, Haifeng Xu, and Yangfeng Ji.
\newblock Cs-shapley: class-wise shapley values for data valuation in
  classification.
\newblock \emph{Advances in Neural Information Processing Systems},
  35:\penalty0 34574--34585, 2022.

\bibitem[Shapley(1953)]{shapley1953value}
Lloyd~S. Shapley.
\newblock A value for n-person games.
\newblock 1953.

\bibitem[Smith and Alvarez(2021)]{smith2021identifying}
Matthew Smith and Francisco Alvarez.
\newblock Identifying mortality factors from machine learning using shapley
  values--a case of covid19.
\newblock \emph{Expert Systems with Applications}, 176:\penalty0 114832, 2021.

\bibitem[Stern and Tettenhorst(2019)]{stern2019hodge}
Ari Stern and Alexander Tettenhorst.
\newblock Hodge decomposition and the shapley value of a cooperative game.
\newblock \emph{Games and Economic Behavior}, 113:\penalty0 186--198, 2019.

\end{thebibliography}

\end{document}